\documentclass{amsart}
\usepackage[latin1]{inputenc}
\usepackage{amsmath}
\usepackage{amsthm}
\usepackage{amsfonts}
\usepackage{amssymb}
\usepackage[T1]{fontenc}
\usepackage{float}
\usepackage[all]{xy}
\usepackage{tikz}
\usepackage{bbm}
\usepackage{color}
\usepackage{hyperref}
\usepackage{stmaryrd}

\theoremstyle{plain}
\newtheorem{theorem}{Theorem}[section]
\newtheorem{prop}[theorem]{Proposition}
\newtheorem{lem}[theorem]{Lemma}
\newtheorem{corol}[theorem]{Corollary}

\theoremstyle{definition}
\newtheorem{defi}[theorem]{Definition}

\newtheorem{rmq}[theorem]{Remark}
\newtheorem{exmp}[theorem]{Example}

\def\<{\left<}
\def\>{\right>}

\def\d{{\partial}}
\def\k{\mathbf{k}}

\def\ens#1{\left\{ #1 \right\}}

\def\fl{{\longrightarrow}\,}

\def\TT{{\mathcal{T}}}

\def\A{{\mathbf{A}}}

\def\E{{\mathbf{E}}}
\def\T{{\mathbf{T}}}
\def\b#1{\overline{#1}}

\def\N{{\mathbb{N}}}

\def\Z{{\mathbb{Z}}}

\def\1{\mathbbm{1}}

\def\x{{\mathbf x}}

\def\ens#1{\left\{ #1 \right\}}

\def\im{{\rm{im}\,}}

\def\1{\mathbbm{1}}

\def\M{\mathbf M}

\def\MCG{\mathrm{MCG}}

\def\Prufer#1#2#3#4{
	\coordinate (x) at (#1,#2);
	\fill (x) circle (.1);
	\draw[#4] (x) .. controls (#1,#2+2.5) and (#1+.5,#2+3) .. (#1+#3,#2+3);
}

\def\adic#1#2#3#4{
	\coordinate (x) at (#1,#2);
	\fill (x) circle (.1);
	\draw[#4] (x) .. controls (#1,#2+2.5) and (#1-.5,#2+3) .. (#1-#3,#2+3);
}

\def\hadic#1#2#3#4{
	\coordinate (x) at (#1,#2);
	\fill (x) circle (.1);
	\draw[#4] (x) .. controls (#1,#2-2.5) and (#1+.5,#2-3) .. (#1+#3,#2-3);
}

\def\hPrufer#1#2#3#4{
	\coordinate (x) at (#1,#2);
	\fill (x) circle (.1);
	\draw[#4] (x) .. controls (#1,#2-2.5) and (#1-.5,#2-3) .. (#1-#3,#2-3);
}

\title[Compactifying Exchange Graphs]{Compactifying Exchange Graphs I, \\ Annuli and tubes} 
\author{K.~Baur and G.~Dupont}
\address{Karl-Franzens-Universit\"at Graz, Institute for Mathematics and Scientific Computing, Heinrichstrasse 36, 8010 Graz, Austria}
\email{karin.baur@uni-graz.at}
\address{Universit\'e des Antilles et de la Guyane, IUFM de Guadeloupe, Morne Ferret, 97178 Les Abymes cedex, France.}
\email{gdupont@iufm.univ-ag.fr}
\date{\today}

\begin{document}

\begin{abstract}
	We introduce the notion of an \emph{asymptotic triangulation} of the annulus. We show that asymptotic triangulations can be mutated as the usual triangulations and describe their exchange graph. 
	Viewing asymptotic triangulations as limits of triangulations under the action of the mapping class group, we compactify the exchange graph of the triangulations of the annulus. The cases of tubes are also considered. % In a further article, we are considering the case of the infinity-gon, cf. \cite{BD:infinity}. 
% \noindent
% {\bf Changes, remarks and questions (Februar 28)}
% \begin{itemize}
% %	\item Concerning your new nice picture of $\d\b\E(C_{2,1})$, I dashed an edge and shaded an area to make it clearer. I also added a remark saying that this particular case is $B_2 \times B_3$.
% %	\item I added a remark saying that the boundary is the same as the exchange graph when $q=0$.
% 	\item I do not know what we should say about these polytopes we have.
% \end{itemize}
\end{abstract}

\maketitle

\setcounter{tocdepth}{1}
\tableofcontents

\section*{Introduction}
	The \emph{mutation} of a triangulation in a marked surfaces is a combinatorial construction which plays a prominent role in Teichm\"uller theory, as it first appeared in the seminal works of Penner \cite{Penner:lambda,Penner:bordered} and Fock \cite{Fock:dual}. Inspired by these ideas, Fomin, Shapiro and Thurston initiated a systematic study of these mutations for arbitrary marked surfaces using the concept of cluster combinatorics, see \cite{FST:surfaces, FT:surfaces2}. This led in particular to the notion of \emph{exchange graph of triangulations} of a marked surface.

	Cluster structures also present a very strong interplay with representation theory of algebras. From this perspective, the exchange graphs encode the combinatorics of the (cluster-)tilting objects in suitable triangulated 2-Calabi-Yau categories, see for instance \cite{Reiten:TiltingCluster}. In the particular case of the so-called \emph{tubes}, Buan and Krause enriched this combinatorial complex by considering tilting objects in an appropriate completion of the category \cite{BuanKrause:cotilting}. 
	
	The aim of this article is to propose a combinatorial version of this completion at the level of triangulations of marked surfaces, and more precisely in the context of unpunctured annuli and tubes (the latter being unpunctured annuli with marked points on one boundary component only). 
	
	We therefore introduce a notion of \emph{asymptotic triangulation} of an unpunctured marked surface which enriches the complex of triangulations. We prove that these asymptotic triangulations can be mutated as the usual triangulations. This leads to a new polytope associated with the surface, called \emph{the boundary of the exchange graph}. We finally prove that this boundary can naturally be viewed as the boundary of a compactification of the usual exchange graph of triangulations under the action of the mapping class group.
	
% 	The case of the infinity-gon will be considered in a forthcoming article, see \cite{BD:infinity}. 

\section*{Notations}
	Let $p,q \geq 0$ and let $C_{p,q}$ denote the annulus with $p$ points marked on a boundary component and $q$ points marked on the other. We denote by $\M$ the set of marked points on the boundary of $C_{p,q}$. Let $z$ denote a non-contractible closed curve in the interior of $C_{p,q}$ so that $\pi_1(C_{p,q}) \simeq \Z = \<z\>$.

	Let $D_z$ denote the positive \emph{Dehn twist} with respect to $z$ and $\MCG(C_{p,q}) = \<D_z\> \simeq \Z$ be the \emph{mapping class group} of $C_{p,q}$.

	Two curves in $C_{p,q}$ are called \emph{compatible} if they do not intersect except possibly at a common extremity which is a marked point. 
	Two isotopy classes of curves are called \emph{compatible} if there exist compatible representatives. 

	We let $\A(C_{p,q})$ denote the set of isotopy classes of (non-boundary) \emph{arcs} in $C_{p,q}$, that is, of curves joining two marked points which are compatible with themselves and which do not connect neighbouring marked points on the boundary. 
	Abusing terminology, an element in $\A(C_{p,q})$ is called an arc. We let $\T(C_{p,q})$ denote the set of triangulations of $C_{p,q}$, that is, the set of maximal collections of pairwise distinct compatible arcs.

	If $p,q \geq 1$, given a triangulation $T \in \T(C_{p,q})$, we denote by $Q_T$ the corresponding quiver, see \cite{FST:surfaces}.

	$\k$ is a field.
	
\section{Asymptotic triangulations of the annulus}\label{section:annulus}
		We first assume $p,q\ge 1$ so that $C_{p,q}$ is an unpunctured marked surface in the sense of \cite{FST:surfaces}; the case $q=0$ will be considered later (cf. Remark \ref{rmq:q=0}). We denote by $\d$ and $\d'$ its two boundary components such that $\d$ contains $p$ marked points and $\d'$ contains $q$ marked points. The orientation on $C_{p,q}$ induces an orientation on each boundary component. Following these orientations, we denote by $m_1, \ldots, m_p$ the marked points on $\d$ and by $m'_1, \ldots, m'_q$ the marked points on $\d'$.

		\begin{defi}[Peripheral and bridging arcs]
			An arc in $\A(C_{p,q})$ is called \emph{peripheral} if its two extremities are on the same boundary components. It is called \emph{bridging} otherwise. 
		\end{defi}

	Given a pair $\ens{m_1,m_2}$ of marked points on the same boundary component, there exist at most two peripheral arcs with these marked points as extremities. We denote by $(m_1,m_2)$ the unique arc which is homotopic to the boundary component joining $m_1$ to $m_2$ in this order following the orientation of the boundary, see Figure \ref{fig:arcsCpq}.

	\begin{defi}[Pr\"ufer and adic curves]
		Given a marked point $m$ in $C_{p,q}$, we let~:
		\begin{itemize}
			\item $\pi_m$ be the isotopy class of the simple curve starting at $m$ and spiraling positively around the annulus, the \emph{Pr\"ufer curve} at $m$;
			\item $\alpha_m$ be the isotopy class of the simple curve starting at $m$ and spiraling negatively around the annulus, the \emph{adic curve} at $m$.
		\end{itemize}

		The set of \emph{asymptotic arcs} is
		$$\overline \A(C_{p,q}) = \A(C_{p,q}) \sqcup \ens{\pi_m,\alpha_m \ | \ m \in \M}.$$
		
		An asymptotic arc is called \emph{strictly asymptotic} if it does not belong to $\A(C_{p,q})$.
	\end{defi}

	\begin{figure}[htb]
		\begin{center}
			\begin{tikzpicture}[scale = .5]
				\tikzstyle{every node} = [font = \small]
				\foreach \x in {0}
				{
					\foreach \y in {-8}
					{
						\draw[<-] (\x-5,\y+4) -- (\x+5,\y+4);
						\fill (\x-6,\y+4) node {$\d'$};
						\draw[->] (\x-5,\y-4) -- (\x+5,\y-4);
						\fill (\x-6,\y-4) node {$\d$};

						\foreach \t in {-4,-3,...,4}
						{
							\fill (\x+\t,\y+4) circle (.1);
							\fill (\x+\t,\y-4) circle (.1);
						}

						\fill (\x-4,\y+4) node [above] {$m'_1$};
						\fill (\x-3,\y+4) node [above] {$m'_{q}$};
						\fill (\x-2,\y+4) node [above] {$\cdots$};
						\fill (\x-1,\y+4) node [above] {$\cdots$};
						\fill (\x,\y+4) node [above] {$\cdots$};
						\fill (\x+1,\y+4) node [above] {$\cdots$};
						\fill (\x+2,\y+4) node [above] {$\cdots$};
						\fill (\x+3,\y+4) node [above] {$m'_2$};
						\fill (\x+4,\y+4) node [above] {$m'_1$};

						\fill (\x-4,\y-4) node [below] {$m_1$};
						\fill (\x-3,\y-4) node [below] {$m_2$};
						\fill (\x-2,\y-4) node [below] {$\cdots$};
						\fill (\x-1,\y-4) node [below] {$\cdots$};
						\fill (\x,\y-4) node [below] {$\cdots$};
						\fill (\x+1,\y-4) node [below] {$\cdots$};
						\fill (\x+2,\y-4) node [below] {$\cdots$};
						\fill (\x+3,\y-4) node [below] {$m_p$};
						\fill (\x+4,\y-4) node [below] {$m_1$};

						\hadic{\x+4}{\y+4}{2}{}
						\fill (\x,\y+1) node [above] {$\alpha_{m'_1}$};
						\hadic{\x-4}{\y+4}{9}{}

						\Prufer{\x-4}{\y-4}{9}{}
						\fill (\x,\y-1) node [below] {$\pi_{m_1}$};
						\Prufer{\x+4}{\y-4}{1}{}

						\draw[dashed,gray] (\x-4,\y+4) -- (\x-4,\y-4);
						\draw[dashed,gray] (\x+4,\y+4) -- (\x+4,\y-4);

						\draw (\x-3,\y+4) .. controls (\x-2,\y+3) and (\x+2,\y+3) .. (\x+3,\y+4);
						\fill (\x,\y+3.3) node [below] {\tiny $(m'_2,m'_q)$};

						\draw (\x-3,\y-4) .. controls (\x-2,\y-3) and (\x+2,\y-3) .. (\x+3,\y-4);
						\fill (\x,\y-2.7) node [] {\tiny $(m_2,m_p)$};
					}
				}
			\end{tikzpicture}
		\end{center}
		\caption{Asymptotic arcs in the annulus.}\label{fig:arcsCpq}
	\end{figure}
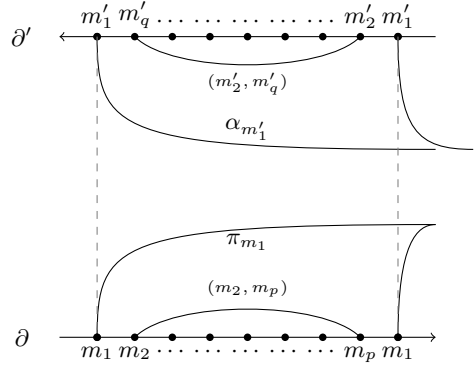

With these notions, we can now define asymptotic triangulations of the annulus. 

		\begin{defi}[Asymptotic triangulations of the annulus]
			An \emph{asymptotic triangulation of the annulus} is a maximal collection of pairwise distinct and compatible asymptotic arcs in $\overline \A(C_{p,q})$. We denote by $\b\T(C_{p,q})$ the set of asymptotic triangulations.

			An asymptotic triangulation $T \in \b\T(C_{p,q})$ is called \emph{strictly asymptotic} if it is not in $\T(C_{p,q})$.
		\end{defi}

		\begin{defi}[Asymptotic exchange graph]
			The \emph{asymptotic exchange graph} is the unoriented graph $\b\E(C_{p,q})$ whose vertices are the asymptotic triangulations in $\b\T(C_{p,q})$ and where two asymptotic triangulations $T,T' \in \b\T(C_{p,q})$ are joined by an edge if and only if they differ by a single asymptotic arc.
		\end{defi}

		Note that a bridging arc is not compatible with any strictly asymptotic arc. Therefore, a strictly asymptotic triangulation does not contain any bridging arcs.

	\subsection{Partial asymptotic triangulations based at a boundary}
		Before we study asymptotic triangulations of the annulus, we focus on partial asymptotic triangulations whose asymptotic arcs all have their extremities on a fixed boundary component.
		
		\begin{defi}[Arc based at a boundary]
			Let $\beta$ be a boundary component of $C_{p,q}$. An asymptotic arc is \emph{based} at $\beta$ if either it is an arc with both its extremities on $\beta$ or if it is strictly asymptotic with its unique extremity on $\beta$.
			
			We denote by $\b\A(\beta)$ the set of asymptotic arcs based at $\beta$ and by $\b\T(\beta)$ the set of maximal collections of pairwise distinct compatible asymptotic arcs in $\b\A(\beta)$. These are called the \emph{partial asymptotic triangulations based at the boundary component $\beta$}.
			
			We denote by $\b\E(\beta)$ the 
exchange graph formed by the partial asymptotic triangulations in $\b\T(\beta)$. 			
%			full subgraph of $\b\E(C_{p,q})$ formed by the partial asymptotic triangulations in $\b\T(\beta)$.
		\end{defi}
		
		\begin{prop}\label{prop:Etube}
			Let $\beta$ be a boundary component of $C_{p,q}$ containing $r$ marked points. Then the following hold:
			\begin{enumerate}
				\item any $T \in \b\T(\beta)$ consists of $r$ asymptotic arcs;
				\item for any $T \in \b\T(\beta)$ and any $\theta \in T$, there exists a unique $\theta^* \in \b\A(\beta)$ such that $\theta^* \neq \theta$ and such that $\mu_{\theta}T := T \setminus \ens{\theta} \sqcup \ens{\theta^*} \in \b\T(\beta)$. The partial asymptotic triangulation $\mu_{\theta}T$ is called the \emph{mutation} of $T$ in $\theta$;
				\item $\b\E(\beta)$ is a connected $r$-regular graph.
			\end{enumerate}
		\end{prop}
		\begin{proof}
			Let $T \in \b\T(\beta)$. We first observe that $T$ necessarily contains a strictly asymptotic arc. 
%Indeed, if not, then there is a marked point $m$ on $\beta$ which is not the vertex of any triangle cut out by $T$. 
Indeed, if not, then there is a marked point $m$ on $\beta$ such that $T$ contains no arc lying above $m$ which forms a triangle with $m$. 
Therefore, as it is illustrated in the figure below, both the Pr\"ufer and the adic curves based at $m$ are compatible with $T$, so that $T$ is not maximal. 
			
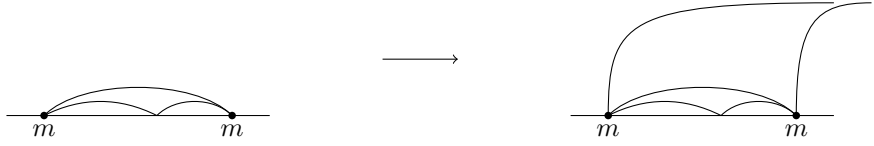
\begin{center}
				\begin{tikzpicture}[scale = .5]
					\foreach \x in {0}
					{
						\draw (\x-3,0) -- (\x+4,0);
				% 		\Prufer{\x-2}{0}{6}{}
						
						\draw[] (\x-2,0) .. controls (\x-1,1) and (\x+2,1) .. (\x+3,0);
						\fill (\x-2,0) circle (.1);
						\fill (\x+3,0) circle (.1);
						
				% 		\Prufer{\x+3}{0}{2}{}
						
						\fill (\x-2,0) node [below] {$m$};
				% 		\fill (\x+1,0) node [below] {$x$};
						\fill (\x+3,0) node [below] {$m$};
						\draw[] (\x-2,0) .. controls (\x-1,.5) and (\x+0,.5) .. (\x+1,0);
						\draw[] (\x+1,0) .. controls (\x+1.5,.5) and (\x+2.5,.5) .. (\x+3,0);
					}
						
						\draw[->] (7,1.5) -- (9,1.5);

					\foreach \x in {15}
					{
						\draw (\x-3,0) -- (\x+4,0);
						\Prufer{\x-2}{0}{6}{}
						
						\draw[] (\x-2,0) .. controls (\x-1,1) and (\x+2,1) .. (\x+3,0);
	% 					\fill (\x+1,0) circle (.1);
						
						\Prufer{\x+3}{0}{2}{}
						
						\fill (\x-2,0) node [below] {$m$};
				% 		\fill (\x+1,0) node [below] {$x$};
						\fill (\x+3,0) node [below] {$m$};
						\draw[] (\x-2,0) .. controls (\x-1,.5) and (\x+0,.5) .. (\x+1,0);
						\draw[] (\x+1,0) .. controls (\x+1.5,.5) and (\x+2.5,.5) .. (\x+3,0);
					}
				\end{tikzpicture}
			\end{center}
			
			We now prove that for any asymptotic arc $\theta \in T$, there exists a unique $\theta^* \in \b\A(\beta)$ such that $\theta^* \neq \theta$ and such that $T \setminus \ens{\theta} \sqcup \ens{\theta^*} \in \b\T(\beta)$. We will distinguish several cases. First assume that $\theta$ is an arc which lies ``below'' another arc $\gamma$, as in the figure below, with $\theta$ the bold arc. 
			
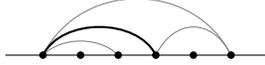
\begin{center}
				\begin{tikzpicture}[scale = .5]
					\foreach \x in {0}
					{
						\foreach \y in {0}
						{
					% 		\fill (\x-2,0) node [below] {$m$};
					% 		\fill (\x+3,0) node [below] {$m$};
							\draw (\x-3,\y) -- (\x+4,0);
							\draw[gray] (\x-2,0) .. controls (\x-1,2) and (\x+2,2) .. (\x+3,0);
							\draw[thick] (\x-2,0) .. controls (\x-1.5,1) and (\x+.5,1) .. (\x+1,0);
							\draw[gray] (\x-2,0) .. controls (\x-1.5,.5) and (\x-.5,.5) .. (\x,0);
							\draw[gray] (\x+1,0) .. controls (\x+1.5,1) and (\x+2.5,1) .. (\x+3,0);
							
							\fill (\x-2,0) circle (.1);
							\fill (\x-1,0) circle (.1);
							\fill (\x,0) circle (.1);
							\fill (\x+1,0) circle (.1);
							\fill (\x+2,0) circle (.1);
							\fill (\x+3,0) circle (.1);
						}
					}
				\end{tikzpicture}
			\end{center}
			Then cutting along $\gamma$ reduces the situation to that of a disc with marked points, in which case the result is known.

			We can now assume that $T$ contains no arc above $\theta$. In case $\theta$ is the unique strictly asymptotic arc in $T$ we let $m$ be the point on $\beta$ where $\theta$ is based. Then if $\theta =\pi_m$, it is clear that $\theta^* = \alpha_m$ and conversely, if $\theta = \alpha_m$, then $\theta^* = \pi_m$, see the figure below.
			
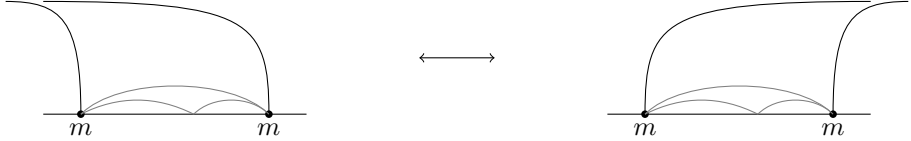
\begin{center}
				\begin{tikzpicture}[scale = .5]
					\foreach \x in {0}
					{
						\draw (\x-3,0) -- (\x+4,0);
						\adic{\x-2}{0}{2}{}
						
						\draw[gray] (\x-2,0) .. controls (\x-1,1) and (\x+2,1) .. (\x+3,0);
	% 					\fill (\x+1,0) circle (.1);
						
						\adic{\x+3}{0}{6}{}
						
						\fill (\x-2,0) node [below] {$m$};
				% 		\fill (\x+1,0) node [below] {$x$};
						\fill (\x+3,0) node [below] {$m$};
						\draw[gray] (\x-2,0) .. controls (\x-1,.5) and (\x+0,.5) .. (\x+1,0);
						\draw[gray] (\x+1,0) .. controls (\x+1.5,.5) and (\x+2.5,.5) .. (\x+3,0);
					}
						
						\draw[<->] (7,1.5) -- (9,1.5);

					\foreach \x in {15}
					{
						\draw (\x-3,0) -- (\x+4,0);
						\Prufer{\x-2}{0}{6}{}
						
						\draw[gray] (\x-2,0) .. controls (\x-1,1) and (\x+2,1) .. (\x+3,0);
	% 					\fill (\x+1,0) circle (.1);
						
						\Prufer{\x+3}{0}{2}{}
						
						\fill (\x-2,0) node [below] {$m$};
				% 		\fill (\x+1,0) node [below] {$x$};
						\fill (\x+3,0) node [below] {$m$};
						\draw[gray] (\x-2,0) .. controls (\x-1,.5) and (\x+0,.5) .. (\x+1,0);
						\draw[gray] (\x+1,0) .. controls (\x+1.5,.5) and (\x+2.5,.5) .. (\x+3,0);
					}
				\end{tikzpicture}
			\end{center}
			
			Now assume that $T$ contains several strictly asymptotic arcs. Then necessarily either all these are Pr\"ufer or all these are adic. By symmetry, let us assume that these are Pr\"ufer. Now there are two cases, either $\theta$ is a strictly asymptotic arc, or it is not. Let first assume that $\theta$ is strictly asymptotic, so that $\theta = \pi_m$ for some marked point $m$ on $\beta$. Let $a$ and $b$ be the two (possibly equal) closest marked points on $\beta$ at which strictly asymptotic arcs of $T \setminus \ens{\pi_m}$ are based. Then, the unique asymptotic arc $\theta^*$ which can replace $\theta$ is the arc joining $a$ and $b$ and bordering a triangle with vertices $a,b$ and $m$, see the figure below.
			
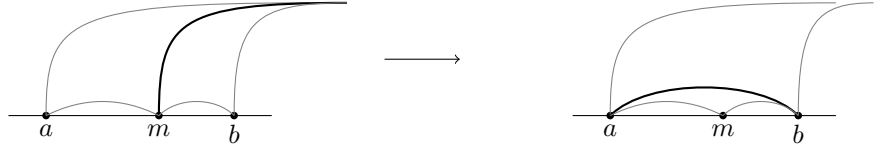
\begin{center}
				\begin{tikzpicture}[scale = .5]
					\foreach \x in {0}
					{
						\draw (\x-3,0) -- (\x+4,0);
						\Prufer{\x-2}{0}{7}{gray}
						\Prufer{\x+1}{0}{5}{thick}
						\Prufer{\x+3}{0}{3}{gray}
						\fill (\x-2,0) node [below] {$a$};
						\fill (\x+1,0) node [below] {$m$};
						\fill (\x+3,0) node [below] {$b$};
						\draw[gray] (\x-2,0) .. controls (\x-1,.5) and (\x+0,.5) .. (\x+1,0);
						\draw[gray] (\x+1,0) .. controls (\x+1.5,.5) and (\x+2.5,.5) .. (\x+3,0);
					}
						
						\draw[->] (7,1.5) -- (9,1.5);

					\foreach \x in {15}
					{
						\draw (\x-3,0) -- (\x+4,0);
						\Prufer{\x-2}{0}{6}{gray}
						
						\draw[thick] (\x-2,0) .. controls (\x-1,1) and (\x+2,1) .. (\x+3,0);
						\fill (\x+1,0) circle (.1);
						
						\Prufer{\x+3}{0}{2}{gray}
						
						\fill (\x-2,0) node [below] {$a$};
						\fill (\x+1,0) node [below] {$m$};
						\fill (\x+3,0) node [below] {$b$};
						\draw[gray] (\x-2,0) .. controls (\x-1,.5) and (\x+0,.5) .. (\x+1,0);
						\draw[gray] (\x+1,0) .. controls (\x+1.5,.5) and (\x+2.5,.5) .. (\x+3,0);
					}
				\end{tikzpicture}
			\end{center}
			Now assume that $\theta$ is not a strictly asymptotic arc. If it is below another arc, then we already saw that the statement holds. Therefore, we can assume that $\theta$ is an arc joining two (possibly equal) points $a$ and $b$ at which strictly asymptotic arcs are based, which are either both Pr\"ufer or both adic. Let us first assume that $T$ contains $\pi_a$ and $\pi_b$. Let $x$ be the vertex such that $\theta$ borders the triangle with vertices $a,b$ and $x$. Then the unique asymptotic arc which is compatible with $T \setminus \ens{\theta}$ and which differs from $\theta$ is $\pi_x$. Similarly, if $T$ contains $\alpha_a$ and $\alpha_b$, then $\theta^* = \alpha_x$. The figure below illustrates the situation.
			
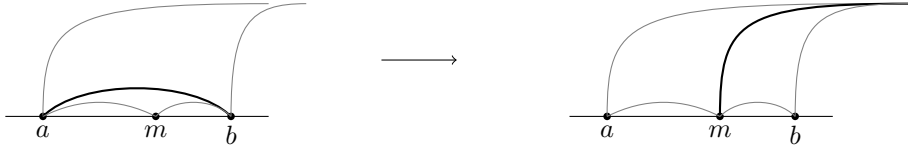
\begin{center}
				\begin{tikzpicture}[scale = .5]
					\foreach \x in {15}
					{
						\draw (\x-3,0) -- (\x+4,0);
						\Prufer{\x-2}{0}{7}{gray}
						\Prufer{\x+1}{0}{5}{thick}
						\Prufer{\x+3}{0}{3}{gray}
						\fill (\x-2,0) node [below] {$a$};
						\fill (\x+1,0) node [below] {$m$};
						\fill (\x+3,0) node [below] {$b$};
						\draw[gray] (\x-2,0) .. controls (\x-1,.5) and (\x+0,.5) .. (\x+1,0);
						\draw[gray] (\x+1,0) .. controls (\x+1.5,.5) and (\x+2.5,.5) .. (\x+3,0);
					}
						
						\draw[->] (7,1.5) -- (9,1.5);

					\foreach \x in {0}
					{
						\draw (\x-3,0) -- (\x+4,0);
						\Prufer{\x-2}{0}{6}{gray}
						
						\draw[thick] (\x-2,0) .. controls (\x-1,1) and (\x+2,1) .. (\x+3,0);
						\fill (\x+1,0) circle (.1);
						
						\Prufer{\x+3}{0}{2}{gray}
						
						\fill (\x-2,0) node [below] {$a$};
						\fill (\x+1,0) node [below] {$m$};
						\fill (\x+3,0) node [below] {$b$};
						\draw[gray] (\x-2,0) .. controls (\x-1,.5) and (\x+0,.5) .. (\x+1,0);
						\draw[gray] (\x+1,0) .. controls (\x+1.5,.5) and (\x+2.5,.5) .. (\x+3,0);
					}
				\end{tikzpicture}
			\end{center}
			
			Therefore, we saw that in either case there is a unique asymptotic arc $\theta^* \neq \theta$ such that $T \setminus \ens{\theta} \sqcup \ens{\theta^*} \in \b\T(\beta)$, proving the second point.
			
			We now prove that $\b\E(\beta)$ is connected by proving that any $T \in \b\T(\beta)$ is connected to the partial asymptotic triangulation consisting only of Pr\"ufer curves, one based at each marked point. For this we fix $T \in \b\T(\beta)$ and according to the above discussion, we can successively mutate $T$ at its strictly asymptotic arcs until we get a partial asymptotic triangulation $T'$ containing exactly one strictly asymptotic arc, this arc being based at a certain marked point $x$. Now either $\pi_x \in T'$ and we set $T'' = T'$, or $\pi_x \in T'$ and we set $T'' = T' \setminus \ens{\alpha_x} \sqcup \ens{\pi_x}$ so that in both cases $T''$ contains exactly one strictly asymptotic arc, which is Pr\"ufer. Now according to the previous discussion, we can successively mutate at the peripheral arcs of $T''$ in such a way that at each step we create precisely one new Pr\"ufer curve, this being done until we end up with the partial asymptotic triangulation consisting only of Pr\"ufer curves, one based at each marked point, as claimed. Therefore $\b\E(\beta)$ is connected. The picture below illustrates the successive steps.
			
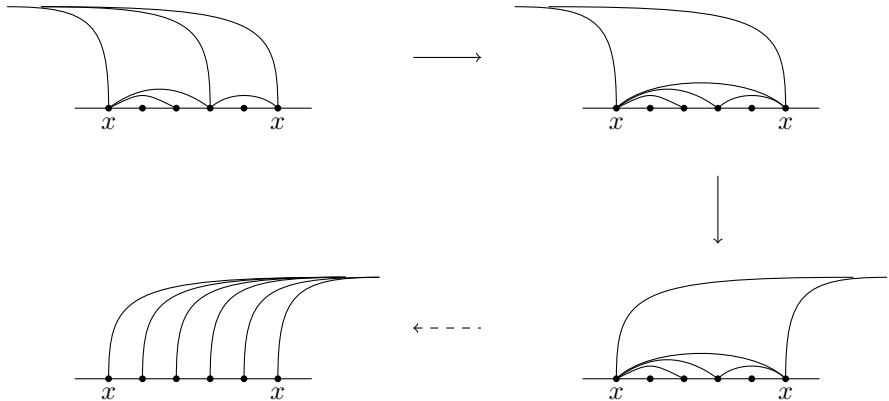
\begin{center}
				\begin{tikzpicture}[scale = .45]
					\foreach \x in {0}
					{
						\foreach \y in {0}
						{
					% 		\fill (\x-2,0) node [below] {$m$};
					% 		\fill (\x+3,0) node [below] {$m$};
							\draw (\x-3,\y+0) -- (\x+4,\y+0);
							\adic{\x-2}{\y+0}{3}{}
							\adic{\x+1}{\y+0}{5}{}
							\adic{\x+3}{\y+0}{7}{}
							\fill (\x-1,\y+0) circle (.1);
							\fill (\x+0,\y+0) circle (.1);
							\fill (\x+2,\y+0) circle (.1);
							\fill (\x-2,\y+0) node [below] {$x$};
							\fill (\x+3,\y+0) node [below] {$x$};
							\draw[] (\x-2,\y+0) .. controls (\x-1,\y+.75) and (\x+0,\y+.75) .. (\x+1,\y+0);
							\draw[] (\x-2,\y+0) .. controls (\x-1,\y+.5) and (\x-1,\y+.5) .. (\x,\y+0);
							\draw[] (\x+1,\y+0) .. controls (\x+1.5,\y+.5) and (\x+2.5,\y+.5) .. (\x+3,\y+0);
						}
					}

					\draw[->] (7,1.5) -- (9,1.5);
					
					\foreach \x in {15}
					{
						\foreach \y in {0}
						{
					% 		\fill (\x-2,0) node [below] {$m$};
					% 		\fill (\x+3,0) node [below] {$m$};
							\draw (\x-3,\y+0) -- (\x+4,\y+0);
							\adic{\x-2}{\y+0}{3}{}
							\adic{\x+3}{\y+0}{7}{}
							\fill (\x-2,\y+0) node [below] {$x$};
							\fill (\x+3,\y+0) node [below] {$x$};
							\fill (\x-1,\y+0) circle (.1);
							\fill (\x,\y+0) circle (.1);
							\fill (\x+1,\y+0) circle (.1);
							\fill (\x+2,\y+0) circle (.1);
							\draw[] (\x-2,\y+0) .. controls (\x-1,\y+1) and (\x+2,\y+1) .. (\x+3,\y+0);
							\draw[] (\x-2,\y+0) .. controls (\x-1,\y+.75) and (\x+0,\y+.75) .. (\x+1,\y+0);
							\draw[] (\x-2,\y+0) .. controls (\x-1,\y+.5) and (\x-1,\y+.5) .. (\x,\y+0);
							\draw[] (\x+1,\y+0) .. controls (\x+1.5,\y+.5) and (\x+2.5,\y+.5) .. (\x+3,\y+0);
						}
					}
					
					\draw[->] (16,-2) -- (16,-4);
					
					\foreach \x in {15}
					{
						\foreach \y in {-8}
						{
					% 		\fill (\x-2,0) node [below] {$m$};
					% 		\fill (\x+3,0) node [below] {$m$};
							\draw (\x-3,\y+0) -- (\x+4,\y+0);
							\Prufer{\x-2}{\y+0}{7}{}
							\Prufer{\x+3}{\y+0}{3}{}
							\fill (\x-2,\y+0) node [below] {$x$};
							\fill (\x+3,\y+0) node [below] {$x$};
							\fill (\x-1,\y+0) circle (.1);
							\fill (\x,\y+0) circle (.1);
							\fill (\x+1,\y+0) circle (.1);
							\fill (\x+2,\y+0) circle (.1);
							\draw[] (\x-2,\y+0) .. controls (\x-1,\y+1) and (\x+2,\y+1) .. (\x+3,\y+0);
							\draw[] (\x-2,\y+0) .. controls (\x-1,\y+.75) and (\x+0,\y+.75) .. (\x+1,\y+0);
							\draw[] (\x-2,\y+0) .. controls (\x-1,\y+.5) and (\x-1,\y+.5) .. (\x,\y+0);
							\draw[] (\x+1,\y+0) .. controls (\x+1.5,\y+.5) and (\x+2.5,\y+.5) .. (\x+3,\y+0);
						}
					}

					\draw[dashed,<-] (7,-6.5) -- (9,-6.5);
					
					\foreach \x in {0}
					{
						\foreach \y in {-8}
						{
					% 		\fill (\x-2,0) node [below] {$m$};
					% 		\fill (\x+3,0) node [below] {$m$};
							\draw (\x-3,\y+0) -- (\x+4,\y+0);
							\Prufer{\x-2}{\y+0}{7}{}
							\Prufer{\x-1}{\y+0}{6}{}
							\Prufer{\x}{\y+0}{5}{}
							\Prufer{\x+1}{\y+0}{4}{}
							\Prufer{\x+2}{\y+0}{4}{}
							\Prufer{\x+3}{\y+0}{3}{}
							\fill (\x-2,\y+0) node [below] {$x$};
							\fill (\x+3,\y+0) node [below] {$x$};
						}
					}
				\end{tikzpicture}
			\end{center}
			
			Now it only remains to prove that $T \in \b\E(\beta)$ contains $r$ asymptotic arcs where $r$ is the number of marked points on $\beta$. By connectedness of $\b\E(\beta)$, the number of asymptotic arcs in a partial asymptotic triangulation $T \in \b\T(\beta)$ does not depend on the choice of $T$. Therefore it is enough to observe that the partial asymptotic triangulation consisting of the Pr\"ufer curves based at each point contains as many asymptotic arcs as the number of marked points, proving the first point. Together with the second point, this proves that $\b\E(\beta)$ is $r$-regular, and thus the proposition.
		\end{proof}

	\subsection{From partial to complete asymptotic triangulations}%\label{section:partial-asympt}

		\begin{lem}\label{lem:twoarcs}
			Let $T \in \b\T(C_{p,q})$ be an asymptotic triangulation. Then the following hold:
			\begin{enumerate}
				\item If $T \in \T(C_{p,q})$, then $T$ contains at least two bridging arcs;
				\item If $T$ is strictly asymptotic, then it contains at least two strictly asymptotic arcs and there are partial asymptotic triangulations $T_{\d}$ and $T_{\d'}$ based respectively at $\d$ and $\d'$ such that $T = T_{\d} \sqcup T_{\d'}$.
			\end{enumerate}
		\end{lem}
		\begin{proof}
			Let $T \in \T(C_{p,q})$ and let $\b T$ be the set of all bridging arcs in $T$. Then there exists at least one point on each boundary component which is not covered by an arc of $\b T$ forming a triangle with this point as a vertex (as the point $m$ in the first picture of the proof of Proposition~\ref{prop:Etube}). In this situation, completing $\b T$ to a triangulation amounts to triangulating another annulus with less marked point on each boundary component using only bridging arcs, this always requires at least two bridging arcs.

			Now assume that $T$ is a strictly asymptotic triangulation so that it contains a strictly asymptotic arc $\gamma_m$ based at a point $m$ on one of the boundary component, say $\d$. Since $T$ consists only of peripheral arcs and of strictly asymptotic arcs, the asymptotic arcs based at marked points on the other boundary component $\d'$ are either peripheral or strictly asymptotic. If such an asymptotic arc is strictly asymptotic, the first part of claim (2) is proved. Therefore, assume that all these arcs are peripheral. Then, as in the previous discussion, there exists a point $m'$ on $\d'$ above which there lies no triangle cut out by peripheral arcs of $T$ (as the point $m$ in the first picture of the proof of Proposition~\ref{prop:Etube}). As the strictly asymptotic arcs based at $m'$ are compatible with all the peripheral arcs and with all the strictly asymptotic arcs based at marked points on $\d$, the asymptotic triangulation $T$ necessarily contains a strictly asymptotic arc based at $m'$, which proves the first part of claim (2). Let $T_{\d}$ be the set of asymptotic arcs of $T$ based at $\d$ and by $T_{\d'}$ be the set of asymptotic arcs of $T$ based at $\d'$. Since $T$ contains no bridging arcs, it follows that $T = T_{\d} \sqcup T_{\d'}$, proving the second point.
		\end{proof}

		\begin{corol}\label{corol:nbarcs}
			The number of asymptotic arcs in an asymptotic triangulation of $C_{p,q}$ is $p+q$.
		\end{corol}
		\begin{proof}
			Let $T$ be an asymptotic triangulation. Either it is a triangulation in the usual sense in which case it is known that is has $p+q$ arcs, or it is strictly asymptotic in which case we can decompose it as $T = T_{\d} \sqcup T_{\d'}$ and the result then follows from Lemma \ref{lem:twoarcs}, observing that $|T_{\d}|=p$ and $|T_{\d'}|=q$.
		\end{proof}
		
		\begin{prop}\label{prop:flips}
			Let $T \in \b\T(C_{p,q})$ and let $\theta \in T$. Then there exists a unique $\theta^* \in \b\A(C_{p,q})$ such that $\theta^* \neq \theta$ and such that $T \setminus \ens{\theta} \sqcup \ens{\theta^*} \in \b\T(C_{p,q})$. This is called the \emph{mutation} of $T$ in $\theta$, written as $\mu_{\theta}T$. 
		\end{prop}
		\begin{proof}
			Let $\T \in \b\T(C_{p,q})$. If $T \in \T(C_{p,q})$, then it is well-known that for any $\theta \in T$, there exists a unique $\theta^* \in \A(C_{p,q})$ such that $\theta^* \neq \theta$ and such that $T \setminus \ens{\theta} \sqcup \ens{\theta^*} \in \T(C_{p,q})$. Now it follows from Lemma \ref{lem:twoarcs} that every triangulation of $C_{p,q}$ contains at least two bridging arcs so that $T \setminus \ens{\theta}$ is not compatible with any strictly asymptotic arc. Therefore, the above arc $\theta^*$ is the unique asymptotic arc such that $T \setminus \ens{\theta} \sqcup \ens{\theta^*} \in \b\T(C_{p,q})$.

			Assume now that $T$ is strictly asymptotic. Then it follows from Lemma \ref{lem:twoarcs} that $T = T_{\d} \sqcup T_{\d'}$ where $\d$ and $\d'$ are the two boundary components of $C_{p,q}$. By symmetry, assume that $\theta \in T_{\d}$. By maximality of $T_{\d'}$, any element in $\b\A(\d')$ which is compatible with $T_{\d'}$ is already contained in $T_{\d'}$. Moreover, since $T_{\d'}$ contains a strictly asymptotic arc, any triangulation containing this strictly asymptotic arc cannot contain a bridging arc, and therefore if $\theta^*$ exists, it is necessarily based at $\d$. Then it follows from Proposition \ref{prop:Etube} that there exists a unique $\theta^* \in \b\A(\d)$ such that $\mu_{\theta}T_{\d} = T_{\d} \setminus \ens{\theta} \sqcup \ens{\theta^*} \in \b\T(\d)$ and therefore, $\mu_{\theta}T_{\d} \sqcup T_{\d'} \in \b\T(C_{p,q})$.
		\end{proof}

%	\subsection{Asymptotic triangulations of tubes}\label{section:tubes}

		\begin{rmq}\label{rmq:q=0}
			In case $p\ge 1$ we call $C_{p,0}$ a \emph{tube with $p$ marked points}. 
			As usual, we denote by $\d$ the boundary component of $C_{p,0}$ containing $p$ marked points.
			An {\em asymptotic arc in $C_{p,0}$} is an element in $\b\A(C_{p,0}) = \b\A(\d)$, an {\em asymptotic triangulation of $C_{p,0}$} is an element in $\b\T(C_{p,0}) = \b\T(\d)$, and the \emph{asymptotic exchange graph of $C_{p,0}$} is $\b\E(C_{p,0}) = \b\E(\d)$.

			In particular, the statements of Proposition~\ref{prop:Etube} hold for the tube, that is~:
			\begin{enumerate}
				\item Any asymptotic triangulation $T \in \b\T(C_{p,0})$ consists of $p$ asymptotic arcs. 
				\item For any $T \in \b\T(C_{p,0})$ and any $\theta \in T$, there exists a unique $\theta^* \neq \theta$ in $\b\A(C_{p,0})$ such that 
					$\mu_{\theta}(T)=T \setminus \{\theta\} \sqcup \{\theta^*\}$ is an asymptotic triangulation of $C_{p,0}$. 
				\item $\b\E(C_{p,0})$ is a $p$-regular graph. \hfill \qed
			\end{enumerate}
		\end{rmq}

	\subsection{The boundary of the exchange graph}
		We next move on to the general situation of the annulus $C_{p,q}$ and we now allow the case $p \geq 1, q=0$ as well. As we saw in Propositions \ref{prop:Etube} and \ref{prop:flips}, the mutation of a strictly asymptotic triangulation of $C_{p,q}$ is again strictly asymptotic. Therefore, the full subgraph of the asymptotic exchange graph consisting of the strictly asymptotic triangulations is a union of connected components of the asymptotic exchange graph. This graph is called the \emph{boundary of the asymptotic exchange graph} and is denoted by $\d\b\E(C_{p,q})$.
		
		In case $q = 0$, any asymptotic triangulation is strictly asymptotic so that we have $\d\b\E(C_{p,0}) \simeq \b\E(C_{p,0})$ for any $p \geq 1$.

		For $p \geq 1$ we denote by $\mathcal T_p$ the mesh category over $\k$ of the stable translation quiver $\mathbb{Z}A_{\infty}/(p)$. Slightly abusing the terminology, we also call $\mathcal T_p$ a tube of rank $p$. We extend this category by adding filtered direct limits and filtered inverse limits of modules, see \cite{BBM:torsiontubes}. 
		We denote the resulting category by $\overline{\mathcal T_p}$, called the \emph{extended tube of rank $p$}. Its indecomposable objects are the indecomposable objects of $\mathcal T_p$ together with the Pr\"ufer modules $M_{i,\infty}$, for $i=0,\dots,p-1$ and the adic modules $\widehat{M}_i$, for $i=0,\dots,p-1$. 

		We write $\E(\overline{\mathcal T}_p)$ for the exchange graph of maximal rigid objects of the extended tube $\overline{\mathcal T}_p$. 
		This exchange graph has been first studied in  \cite[Appendix B]{BuanKrause:cotilting}. It is called the {\em circular associahedron}. 
%\footnote{What would be a reference for this exchange graph~? \cite{BBM:torsiontubes}~?} 

		\begin{lem} \label{lem:exchange}
			Let $\beta$ be a boundary component with $p>0$ marked points. Then 
			$$\b\E(\beta) \simeq \E(\overline{\mathcal T}_p).$$ 
		\end{lem}
		\begin{proof}
			The assertion follows, since the set of asymptotic triangulations of the annulus $C_{p,0}$ is in bijection with the set of isomorphism classes of maximal rigid objects in the extended tube $\overline{\mathcal T_p}$ and that this bijection commutes with mutations, see \cite{BBM:torsiontubes}.
		\end{proof}
We thus obtain the following geometric analogue of \cite[Corollary 4.9]{BBM:torsiontubes}: 
		\begin{corol}
			There are $2{2p-1\choose p-1}$ asymptotic triangulations of the annulus $C_{p,0}$. 
		\end{corol}
		\begin{proof}
			By \cite[Theorem 5.2]{BuanKrause:cotilting}, there is an isomorphism between the poset of basic maximal rigid objects of
			$\overline{\mathcal T_p}$ containing a Pr\"ufer object and a poset $\tilde{\mathcal C}(p)$ which 
			is the union of $p$ copies of the Tamari lattice $\mathcal C(p)$. Its graph is called the \emph{circular associahedron} in
			\cite{BuanKrause:cotilting}. The poset $\tilde{\mathcal C}(p)$ has ${2p-1\choose p-1}$ vertices, see \cite[Appendix B]{BuanKrause:cotilting}.
		\end{proof}
		
		\begin{exmp}
			Figure \ref{fig:dbC30} shows the (boundary of the) asymptotic exchange graph $\b\E(C_{3,0})$ of the tube with three marked points.
			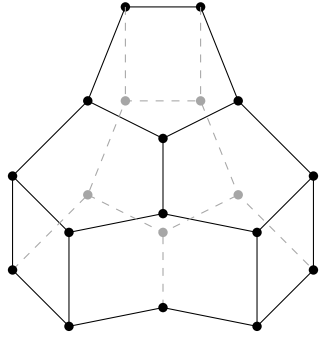
\begin{figure}[htb]
				\begin{center}
					\begin{tikzpicture}[scale = .25]
						\foreach \x in {0}
						{
							\foreach \y in {0}
							{
								\coordinate (A1) at (\x-2,\y+7);
								\coordinate (A2) at (\x+2,\y+7);
								\coordinate (B1) at (\x-4,\y+2);
								\coordinate (B2) at (\x+4,\y+2);
								\coordinate (C) at (\x+0,\y+0);
								\coordinate (D1) at (\x-8,\y-2);
								\coordinate (D2) at (\x+8,\y-2);
								\coordinate (E) at (\x+0,\y-4);
								\coordinate (F1) at (\x-5,\y-5);
								\coordinate (F2) at (\x+5,\y-5);
							
								\draw[] (A1) -- (A2) -- (B2) -- (D2) -- (F2) -- (E) -- (F1) -- (D1) -- (B1) -- cycle;
								\draw[] (B1) -- (C) -- (B2);
								\draw[] (E) -- (C);
							
								\fill (A1) circle (.25);
								\fill (A2) circle (.25);
								
								\fill (B1) circle (.25);
								\fill (B2) circle (.25);
								
								\fill (C) circle (.25);
								
								\fill (D1) circle (.25);
								\fill (D2) circle (.25);
								
								\fill (E) circle (.25);
								
								\fill (F1) circle (.25);
								\fill (F2) circle (.25);
							}
						}
						
						\foreach \x in {0}
						{
							\foreach \y in {-5}
							{
								\coordinate (bA1) at (\x-2,\y+7);
								\coordinate (bA2) at (\x+2,\y+7);
								\coordinate (bB1) at (\x-4,\y+2);
								\coordinate (bB2) at (\x+4,\y+2);
								\coordinate (bC) at (\x+0,\y+0);
								\coordinate (bD1) at (\x-8,\y-2);
								\coordinate (bD2) at (\x+8,\y-2);
								\coordinate (bE) at (\x+0,\y-4);
								\coordinate (bF1) at (\x-5,\y-5);
								\coordinate (bF2) at (\x+5,\y-5);
							
								\draw[dashed,gray!70] (bD1) -- (bB1) -- (bA1) -- (bA2) -- (bB2) -- (bD2);
								\draw[] (bD2) -- (bF2) -- (bE) -- (bF1) -- (bD1);
								\draw[dashed,gray!70] (bB1) -- (bC) -- (bB2);
								\draw[dashed,gray!70] (bE) -- (bC);
								
								\fill[gray!70] (bA1) circle (.25);
								\fill[gray!70] (bA2) circle (.25);
								
								\fill[gray!70] (bB1) circle (.25);
								\fill[gray!70] (bB2) circle (.25);
								
								\fill[gray!70] (bC) circle (.25);
								
								\fill[] (bD1) circle (.25);
								\fill[] (bD2) circle (.25);
								
								\fill[] (bE) circle (.25);
								
								\fill[] (bF1) circle (.25);
								\fill[] (bF2) circle (.25);
							}
						}
						
						\draw[dashed,gray!70] (A1) -- (bA1);
						\draw[dashed,gray!70] (A2) -- (bA2);
			% 			\draw[dashed] (B1) -- (bB1);
			% 			\draw[dashed] (B2) -- (bB2);
						\draw[] (D1) -- (bD1);
						\draw[] (D2) -- (bD2);
						\draw[] (F1) -- (bF1);
						\draw[] (F2) -- (bF2);
					\end{tikzpicture}
				\end{center}
				\caption{The asymptotic exchange graph of $C_{3,0}$, which is isomorphic to its boundary.}\label{fig:dbC30}
			\end{figure}
		\end{exmp}
		
		A direct corollary of the proof of Proposition \ref{prop:flips} and of Lemma \ref{lem:exchange} is the following.
		\begin{corol}
			Let $\d$ and $\d'$ be the boundary components of $C_{p,q}$. Then 
			$$\d\b\E(C_{p,q}) \simeq \b\E(\d) \times \b\E(\d'),\ \mbox{and } \ |\d\b\E(C_{p,q})|=4{2p-1\choose p-1} {2q-1\choose q-1}.$$
		\end{corol}
		In particular, the boundary $\d\b\E(C_{p,q})$ is the product of two circular associahedra. %as in \cite{BuanKrause:cotilting}. 

		\begin{exmp}
			Consider $\b\A(C_{2,1})$. It contains six strictly asymptotic arcs~: one Pr\"ufer and one adic based at one boundary component, two Pr\"ufer and two adics at the other. There are therefore twelve strictly asymptotic triangulations and the boundary $\d\b\E(C_{2,1})$ of the asymptotic exchange graph is depicted in Figure \ref{fig:dbEC21}. Note that in this case, it is a direct product of associahedra of types $B_2$ and $B_3$ respectively.
			
			\begin{figure}
				\begin{tikzpicture}[xscale = .4 ,yscale = .25]
					\draw[thick] (-2,3.5) -- (2,3.5);
					\draw[thick] (-2,-3.5) -- (2,-3.5);
					\hadic{-1}{3.5}{-1}{}
					\hadic{1}{3.5}{-3}{}
					\adic{-1}{-3.5}{-3}{}
					\adic{0}{-3.5}{-2}{}
					\adic{1}{-3.5}{-1}{}
	
					\draw[thick] (-3,0) -- (-4,-2);
					\draw[thick] (3,0) -- (4,-2);
%					\draw[thick] (-3,-16) -- (-4,-14);
%					\draw[thick] (3,-16) -- (4,-14);
					\draw[thick] (0,-6) -- (0,-10);

					\draw[thick] (-2-6,3.5-8) -- (2-6,3.5-8);
					\draw[thick] (-2-6,-3.5-8) -- (2-6,-3.5-8);
					\hadic{-1-6}{3.5-8}{3}{}
					\hadic{1-6}{3.5-8}{1}{}
					\adic{-1-6}{-3.5-8}{-3}{}
					\draw (-7,-11.5) .. controls (-6.5,-10) and (-5.5,-10) .. (-5,-11.5);
					\adic{1-6}{-3.5-8}{-1}{}
					
					\draw[thick] (-2+6,3.5-8) -- (2+6,3.5-8);
					\draw[thick] (-2+6,-3.5-8) -- (2+6,-3.5-8);
					\hadic{-1+6}{3.5-8}{-1}{}
					\hadic{1+6}{3.5-8}{-3}{}
					\adic{0+6}{-3.5-8}{-2}{}
					\draw (6,-11.5) .. controls (6.5,-10) and (7.5,-10) .. (8,-10);
					\draw (4,-10) .. controls (4.5,-10) and (5.5,-10) .. (6,-11.5);
					
					\draw[thick] (-2,3.5-16) -- (2,3.5-16);
					\draw[thick] (-2,-3.5-16) -- (2,-3.5-16);
					\hadic{-1}{3.5-16}{3}{}
					\hadic{1}{3.5-16}{1}{}
					\adic{-1}{-3.5-16}{-3}{}
					\adic{0}{-3.5-16}{-2}{}
					\adic{1}{-3.5-16}{-1}{}

					\draw[dashed,gray!70] (+10,-17) -- (+3,-32);
					\fill[white] (7,-24.5) circle (4);
					\draw[dashed,gray!70] (-10,-17) -- (-3,-32);
					\fill[white] (-7,-24.5) circle (4);
					
					\draw[thick] (-2.5,-17) -- (-3.5,-19);
					\draw[thick] (2.5,-17) -- (3.5,-19);
					\draw[thick] (-6,-13) -- (-6,-17);
					\draw[thick] (6,-13) -- (6,-17);
					\draw[thick] (-9,-11) -- (-10,-13);
					\draw[thick] (-9,-23) -- (-10,-25);
					\draw[thick] (9,-11) -- (10,-13);
					\draw[thick] (9,-23) -- (10,-25);
					\draw[thick] (-13,-21) -- (-13,-25);
					\draw[thick] (13,-21) -- (13,-25);
					\draw[thick,gray!70,dashed] (0,-36) -- (0,-40);
%					\draw[thick,gray] (-10,-17) -- (-3,-32);

					\draw[thick] (-10,-32) -- (-3,-47);
%					\draw[thick,gray] (+10,-17) -- (+3,-32);

					\draw[thick] (+10,-32) -- (+3,-47);
%					\draw (-10,-18) .. controls (-8.5,-28) and (-5.5,-31) .. (-3,-32);

					\draw[thick] (-2-6,-11-8) -- (2-6,-11-8);
					\draw[thick] (-2-6,-18-8) -- (2-6,-18-8);
					\hadic{-1-6}{-11-8}{-1}{}
					\hadic{1-6}{-11-8}{-3}{}
					\adic{-1-6}{-18-8}{-3}{}
					\draw (-7,-26) .. controls (-6.5,-24.5) and (-5.5,-24.5) .. (-5,-26);
					\adic{1-6}{-18-8}{-1}{}

					\draw[thick] (-2+6,-11-8) -- (2+6,-11-8);
					\draw[thick] (-2+6,-18-8) -- (2+6,-18-8);
					\hadic{-1+6}{-11-8}{-1}{}
					\hadic{1+6}{-11-8}{-3}{}
					\adic{0+6}{-18-8}{-2}{}
					\draw (6,-26) .. controls (6.5,-24.5) and (7.5,-24.5) .. (8,-24.5);
					\draw (4,-24.5) .. controls (4.5,-24.5) and (5.5,-24.5) .. (6,-26);

					\draw[thick] (-3-12,3.5-16) -- (1-12,3.5-16);
					\draw[thick] (-3-12,-3.5-16) -- (1-12,-3.5-16);
					\hadic{-2-12}{3.5-16}{-1}{}
					\hadic{0-12}{3.5-16}{-3}{}
					\adic{-2-12}{-3.5-16}{1}{}
					\draw (-14,-19.5) .. controls (-13.5,-18) and (-12.5,-18) .. (-12,-19.5);
					\adic{0-12}{-3.5-16}{3}{}

					\draw[thick] (-3-12,3.5-30) -- (1-12,3.5-30);
					\draw[thick] (-3-12,-3.5-30) -- (1-12,-3.5-30);
					\hadic{-2-12}{3.5-30}{3}{}
					\hadic{0-12}{3.5-30}{1}{}
					\adic{-2-12}{-3.5-30}{1}{}
					\draw (-14,-33.5) .. controls (-13.5,-32) and (-12.5,-32) .. (-12,-33.5);
					\adic{0-12}{-3.5-30}{3}{}

					\draw[thick] (11,3.5-16) -- (15,3.5-16);
					\draw[thick] (11,-3.5-16) -- (15,-3.5-16);
					\hadic{12}{3.5-16}{-1}{}
					\hadic{14}{3.5-16}{-3}{}
					\adic{13}{-3.5-16}{2}{}
					\draw (13,-19.5) .. controls (13.5,-18) and (14.5,-18) .. (15,-18);
					\draw (11,-18) .. controls (11.5,-18) and (12.5,-18) .. (13,-19.5);

					\draw[thick] (11,3.5-30) -- (15,3.5-30);
					\draw[thick] (11,-3.5-30) -- (15,-3.5-30);
					\hadic{12}{3.5-30}{3}{}
					\hadic{14}{3.5-30}{1}{}
					\draw (13,-33.5) .. controls (13.5,-32) and (14.5,-32) .. (15,-32);
					\draw (11,-32) .. controls (11.5,-32) and (12.5,-32) .. (13,-33.5);
					\adic{13}{-3.5-30}{2}{}

					\draw[gray!70,thick] (-2,3.5-30) -- (2,3.5-30);
					\draw[gray!70,thick] (-2,-3.5-30) -- (2,-3.5-30);
					\hadic{-1}{3.5-30}{-1}{gray!70}
					\hadic{1}{3.5-30}{-3}{gray!70}
					\adic{-1}{-3.5-30}{1}{gray!70}
					\adic{0}{-3.5-30}{2}{gray!70}
					\adic{1}{-3.5-30}{3}{gray!70}

					\draw[thick] (-2,3.5-46) -- (2,3.5-46);
					\draw[thick] (-2,-3.5-46) -- (2,-3.5-46);
					\hadic{-1}{3.5-46}{3}{}
					\hadic{1}{3.5-46}{1}{}
					\adic{-1}{-3.5-46}{1}{}
					\adic{0}{-3.5-46}{2}{}
					\adic{1}{-3.5-46}{3}{}

				\end{tikzpicture}
				\caption{The boundary $\d\b\E(C_{2,1})$ of the asymptotic exchange graph of the annulus $C_{2,1}$.}\label{fig:dbEC21}
			\end{figure}
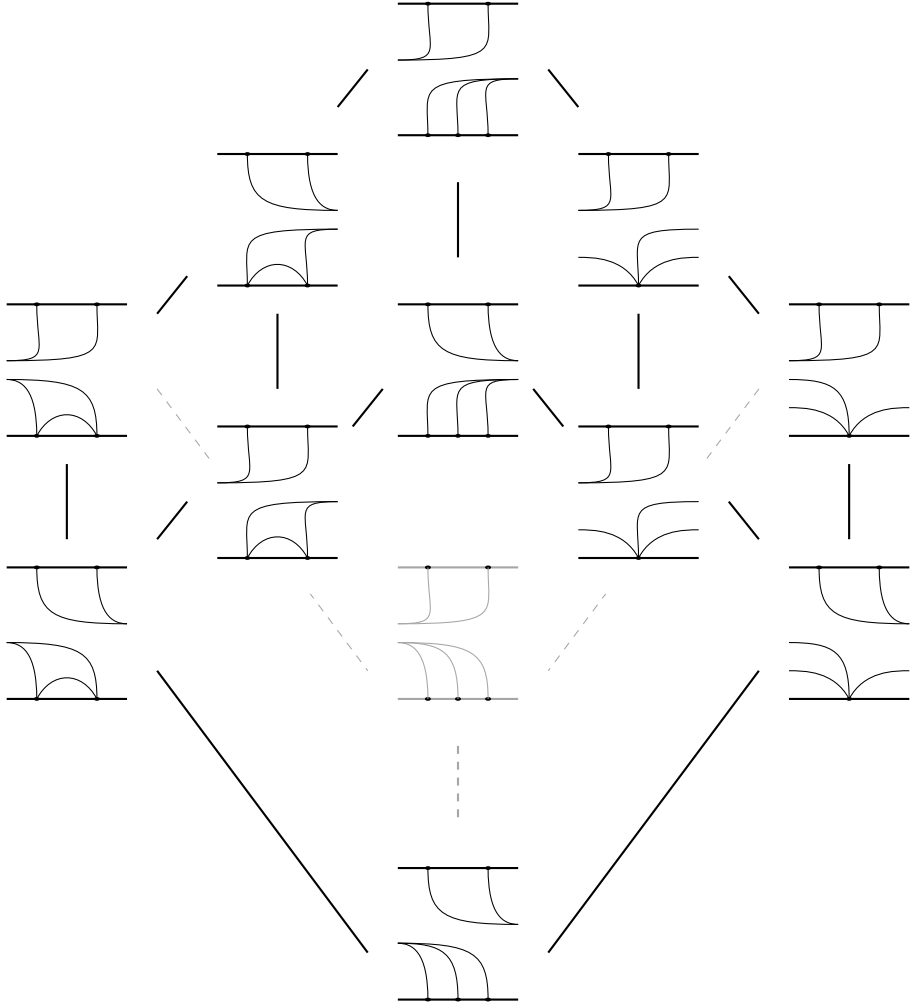
		\end{exmp}

	% 	\subsection{Asymptotic triangulations with generic curves}
	% 	\begin{defi}[Generic curves]
	% 		We denote by $\gamma_z$ (resp. $\gamma_z^{-1}$) the isotopy class of the curve spiraling positively (resp. negatively) around the annulus in both directions. It is called the positive (resp. negative) \emph{generic curve} in $C_{p,q}$.
	% 
	% 		The set of \emph{generic and asymptotic arcs} is 
	% 		$$\overline \A_0(C_{p,q}) = \overline \A(C_{p,q}) \sqcup \ens{\gamma_z,\gamma_z^{-1}}.$$
	% 	\end{defi}

\section{Sequences of triangulations} 
	In this section we prove that asymptotic triangulations do provide a natural way to compactify the usual exchange graph of the triangulations of annulus. 
	
	Throughout this section, we assume that $p,q \geq 1$. Therefore $C_{p,q}$ is an unpunctured marked surface in the sense of \cite{FST:surfaces} and we can associate with any triangulation $T$ of $C_{p,q}$ a quiver $Q_T$ without loops nor 2-cycles, in such a way that mutations of triangulations commute with quiver mutations, see \cite{FST:surfaces} for details.
	
	\subsection{Action of $\MCG(C_{p,q})$ on triangulations}
		Given a quiver $Q$, we denote by 
		$$\TT_Q(C_{p,q}) = \ens{T \in \T(C_{p,q}) \ | \ Q_T \simeq Q}$$
		the set of all triangulations with associated quiver $Q$. 

		\begin{lem}\label{lem:finiteorbits}
			For any quiver $Q$, the mapping class group $\MCG(C_{p,q})$ acts on $\TT_Q(C_{p,q})$ and the quotient $\TT_Q(C_{p,q})/\MCG(C_{p,q})$ is finite.
		\end{lem}
		\begin{proof}
			It is easy to see that $\MCG(C_{p,q})$ acts on $\A(C_{p,q})$ with finitely many orbits. Moreover, as it acts by homeomorphisms, it preserves the triangulations and their quivers so that it acts on $\TT_Q(C_{p,q})$. As $\A(C_{p,q})/\MCG(C_{p,q})$ is finite, so is $\T(C_{p,q})/\MCG(C_{p,q})$ and therefore, so is $\TT_Q(C_{p,q})/\MCG(C_{p,q})$.
		\end{proof}

	\subsection{Converging sequences of triangulations}
		\begin{defi}[Converging sequences]
			Let $(T_n)_{n \geq 0} \subset \T(C_{p,q})$. We say that the sequence $(T_n)_{n \geq 0}$ \emph{converges} if there exists some integer $N \geq 0$ and a monotone map $d: \N \fl \Z$ such that for any $k \geq 0$, we have
			$$T_{N+k} = D_z^{d(k)} \cdot T_N.$$
			\begin{enumerate}
				\item If $\lim_n d(n) = + \infty$ we say that $(T_n)_{n \geq 0}$ \emph{converges positively};
				\item If $\lim_n d(n) = -\infty$ we say that $(T_n)_{n \geq 0}$ \emph{converges negatively};
				\item If $d$ is constant, we say that $(T_n)_{n \geq 0}$ \emph{stabilises (after $N$)}.
			\end{enumerate}

			The map $d$ is called a \emph{convergence} and the triangulation $T_N$ is called a \emph{blueprint} of the converging sequence $(T_n)_{n \geq 0}$.
		\end{defi}

		\begin{prop}
			Any sequence $(T_n)_{n \geq 0} \subset \T(C_{p,q})$ admits a converging subsequence.
		\end{prop}
		\begin{proof}
			It was proved in \cite{FST:surfaces} that the set 
			$$\ens{Q_T \ | \ T \in \T(C_{p,q})}/\sim$$
			is finite where $\sim$ denotes the equivalence relation induced by quiver isomorphisms. Therefore, up to considering a subsequence, we can assume that there is a quiver $Q$ such that 
			$$Q_{T_n} \simeq Q$$
			for any $n \geq 1$.

			According to Lemma \ref{lem:finiteorbits}, $\TT_Q(C_{p,q})/\MCG(C_{p,q})$ is finite. Again, up to considering a subsequence, we can thus assume that there exists some $T \in \T(C_{p,q})$ and a map $\psi : \N \fl \Z$ such that 
			$$T_n = D_z^{\psi(n)} \cdot T$$
			for any $n \geq 1$.

			If $\im(\psi)$ is finite, then there is some $k \in \Z$ such that $\psi^{-1}(k)$ is infinite and therefore, considering an enumeration $(j_n)_{n \geq 0}$ of $\psi^{-1}(k)$, we get $T_{j_n} = D_z^k T$ for any $n \geq 1$. In particular, the subsequence $(T_{j_n})_{n \geq 0}$ stabilises after $j_1$ and therefore converges.

			Now assume that $\im(\psi)$ is infinite and let $(\psi(i_n))_{n \geq 0}$ be a strictly monotone sequence. Then we get $T_{i_n} = D_z^{\psi(i_n)} T$ for any $n \geq 1$. Therefore, $T_{i_n} = D_z^{\psi(i_n)-\psi(i_1)} T_{i_1}$ and $(\psi(i_n)-\psi(i_1))_{n \geq 0}$ is monotone so that $(T_{i_n})_{n \geq 0}$ is a converging subsequence.
		\end{proof}

	\subsection{Limits}
		Let $\gamma:[0,1] \fl C_{p,q}$ be a curve whose isotopy class (also denoted by $\gamma$) is in $\A(C_{p,q})$. If $\gamma(0)$ and $\gamma(1)$ lie on the same boundary component, then $D_z \cdot \gamma \simeq \gamma$ and we set $D_z^{+\infty} \cdot \gamma = \gamma$ and $D_z^{-\infty} \cdot \gamma = \gamma$.
		If $\gamma(0)$ and $\gamma(1)$ do not lie on the same boundary component, we set~:
		$$D_z^{+\infty}\cdot \gamma = \ens{\pi_{\gamma(0)},\pi_{\gamma(1)}}$$
		and 
		$$D_z^{-\infty}\cdot \gamma = \ens{\alpha_{\gamma(0)},\alpha_{\gamma(1)}}.$$

		Given a triangulation $T \in \T(C_{p,q})$, we let 
		$$D_z^{+\infty} \cdot T = \bigcup_{\gamma \in T}D_z^{+\infty} \cdot \gamma\quad \text{and}\quad D_z^{-\infty} \cdot T = \bigcup_{\gamma \in T}D_z^{-\infty} \cdot \gamma.$$

		\begin{defi}[Limit of a triangulation]
			Let $(T_n)_{n \geq 0}$ be a converging sequence of triangulations in $\T(C_{p,q})$ with convergence $d$ and blueprint $T_N$. The \emph{limit} of $(T_n)_{n \geq 0}$ is then defined as 
			$$\lim_{n \fl \infty} T_n = D_z^{\lim_n d(n)} \cdot T_N.$$
		\end{defi}
		Note that this definition does not depend on the choice of the convergence $d$ and of the blueprint $T_N$.

		Figure \ref{fig:exmplimit} shows an example of limits of sequences of triangulations obtained by applying successively the Dehn twist positively or negatively to a given triangulation.
		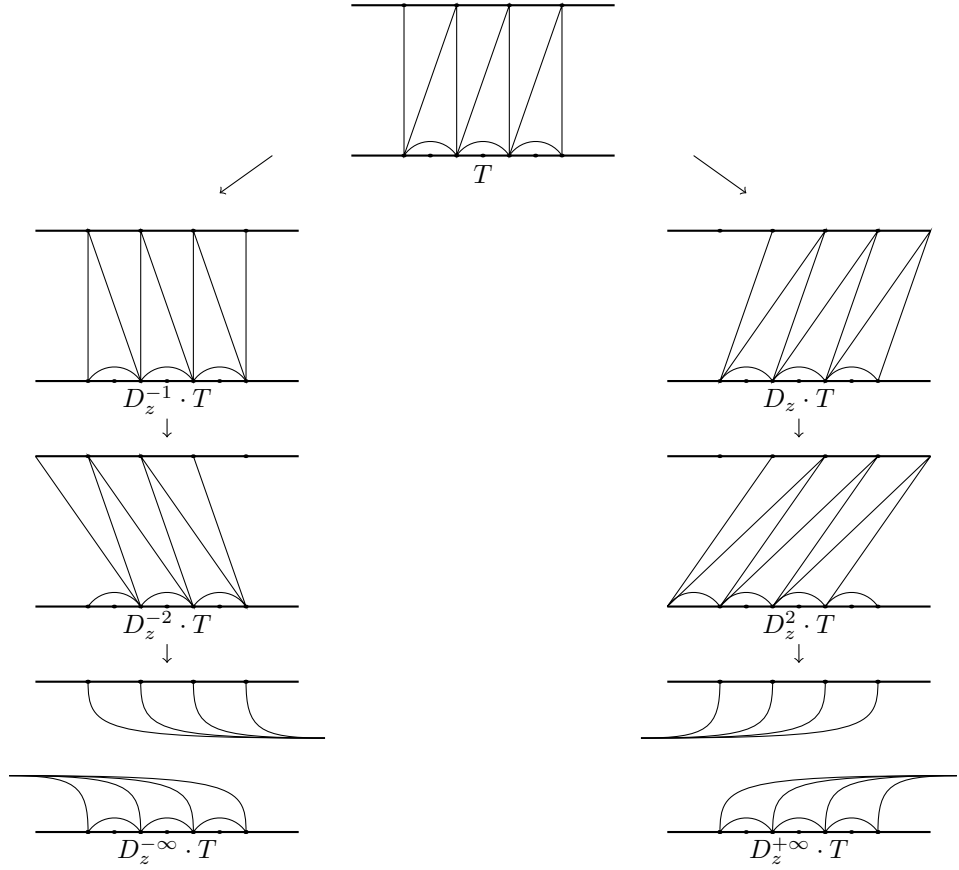
\begin{figure}[htb]
			\begin{center}
				\begin{tikzpicture}[xscale = .35, yscale = .25]
					\foreach \x in {0}
					{
						\foreach \y in {0}
						{
							\coordinate (H1) at (\x-3,\y+4);
							\coordinate (H2) at (\x-1,\y+4);
							\coordinate (H3) at (\x+1,\y+4);
							\coordinate (H4) at (\x+3,\y+4);

							\draw[thick] (\x-5,\y+4) -- (\x+5,\y+4);
							\fill (H1) circle (.1);
							\fill (H2) circle (.1);
							\fill (H3) circle (.1);
							\fill (H4) circle (.1);

							\coordinate (B-1) at (\x-5,\y-4);
							\coordinate (B0) at (\x-4,\y-4);
							\coordinate (B1) at (\x-3,\y-4);
							\coordinate (B2) at (\x-2,\y-4);
							\coordinate (B3) at (\x-1,\y-4);
							\coordinate (B4) at (\x,\y-4);
							\coordinate (B5) at (\x+1,\y-4);
							\coordinate (B6) at (\x+2,\y-4);
							\coordinate (B7) at (\x+3,\y-4);

							\draw[thick] (\x-5,\y-4) -- (\x+5,\y-4);
							\fill (B1) circle (.1);
							\fill (B2) circle (.1);
							\fill (B3) circle (.1);
							\fill (B4) circle (.1);
							\fill (B5) circle (.1);
							\fill (B6) circle (.1);
							\fill (B7) circle (.1);
							
							\draw (H1) -- (B1) -- (H2) -- (B3) -- (H3) -- (B5) -- (H4) -- (B7);
							\draw (B1) .. controls (\x-2.5,\y-3) and (\x-1.5,\y-3) .. (B3);
							\draw (B3) .. controls (\x-.5,\y-3) and (\x+.5,\y-3) .. (B5);
							\draw (B5) .. controls (\x+1.5,\y-3) and (\x+2.5,\y-3) .. (B7);

							\fill (\x,\y-5) node {$T$};
						}
					}

					\draw[->] (-8,-4) -- (-10,-6);
					\draw[->] (8,-4) -- (10,-6);

					\draw[->] (-12,-18) -- (-12,-19);
					\draw[->] (12,-18) -- (12,-19);

					\draw[->] (-12,-30) -- (-12,-31);
					\draw[->] (12,-30) -- (12,-31);

					\foreach \x in {-12}
					{
						\foreach \y in {-12}
						{
							\coordinate (H1) at (\x-3,\y+4);
							\coordinate (H2) at (\x-1,\y+4);
							\coordinate (H3) at (\x+1,\y+4);
							\coordinate (H4) at (\x+3,\y+4);

							\draw[thick] (\x-5,\y+4) -- (\x+5,\y+4);
							\fill (H1) circle (.1);
							\fill (H2) circle (.1);
							\fill (H3) circle (.1);
							\fill (H4) circle (.1);

							\coordinate (B-1) at (\x-5,\y-4);
							\coordinate (B0) at (\x-4,\y-4);
							\coordinate (B1) at (\x-3,\y-4);
							\coordinate (B2) at (\x-2,\y-4);
							\coordinate (B3) at (\x-1,\y-4);
							\coordinate (B4) at (\x,\y-4);
							\coordinate (B5) at (\x+1,\y-4);
							\coordinate (B6) at (\x+2,\y-4);
							\coordinate (B7) at (\x+3,\y-4);

							\draw[thick] (\x-5,\y-4) -- (\x+5,\y-4);
							\fill (B1) circle (.1);
							\fill (B2) circle (.1);
							\fill (B3) circle (.1);
							\fill (B4) circle (.1);
							\fill (B5) circle (.1);
							\fill (B6) circle (.1);
							\fill (B7) circle (.1);
							
							\draw (B1) -- (H1) -- (B3) -- (H2) -- (B5) -- (H3) -- (B7) -- (H4);
							\draw (B1) .. controls (\x-2.5,\y-3) and (\x-1.5,\y-3) .. (B3);
							\draw (B3) .. controls (\x-.5,\y-3) and (\x+.5,\y-3) .. (B5);
							\draw (B5) .. controls (\x+1.5,\y-3) and (\x+2.5,\y-3) .. (B7);

							\fill (\x,\y-5) node {$D_z^{-1} \cdot T$};
						}
					}

					\foreach \x in {-12}
					{
						\foreach \y in {-24}
						{
							\coordinate (H0) at (\x-5,\y+4);
							\coordinate (H1) at (\x-3,\y+4);
							\coordinate (H2) at (\x-1,\y+4);
							\coordinate (H3) at (\x+1,\y+4);
							\coordinate (H4) at (\x+3,\y+4);

							\draw[thick] (\x-5,\y+4) -- (\x+5,\y+4);
							\fill (H1) circle (.1);
							\fill (H2) circle (.1);
							\fill (H3) circle (.1);
							\fill (H4) circle (.1);

							\coordinate (B-1) at (\x-5,\y-4);
							\coordinate (B0) at (\x-4,\y-4);
							\coordinate (B1) at (\x-3,\y-4);
							\coordinate (B2) at (\x-2,\y-4);
							\coordinate (B3) at (\x-1,\y-4);
							\coordinate (B4) at (\x,\y-4);
							\coordinate (B5) at (\x+1,\y-4);
							\coordinate (B6) at (\x+2,\y-4);
							\coordinate (B7) at (\x+3,\y-4);

							\draw[thick] (\x-5,\y-4) -- (\x+5,\y-4);
							\fill (B1) circle (.1);
							\fill (B2) circle (.1);
							\fill (B3) circle (.1);
							\fill (B4) circle (.1);
							\fill (B5) circle (.1);
							\fill (B6) circle (.1);
							\fill (B7) circle (.1);
							
							\draw (H0) -- (B3) -- (H1) -- (B5) -- (H2) -- (B7) -- (H3);
							\draw (B1) .. controls (\x-2.5,\y-3) and (\x-1.5,\y-3) .. (B3);
							\draw (B3) .. controls (\x-.5,\y-3) and (\x+.5,\y-3) .. (B5);
							\draw (B5) .. controls (\x+1.5,\y-3) and (\x+2.5,\y-3) .. (B7);

							\fill (\x,\y-5) node {$D_z^{-2} \cdot T$};
						}
					}

					\foreach \x in {-12}
					{
						\foreach \y in {-36}
						{
							\coordinate (H0) at (\x-5,\y+4);
							\coordinate (H1) at (\x-3,\y+4);
							\coordinate (H2) at (\x-1,\y+4);
							\coordinate (H3) at (\x+1,\y+4);
							\coordinate (H4) at (\x+3,\y+4);

							\draw[thick] (\x-5,\y+4) -- (\x+5,\y+4);
							\fill (H1) circle (.1);
							\fill (H2) circle (.1);
							\fill (H3) circle (.1);
							\fill (H4) circle (.1);

							\coordinate (B-1) at (\x-5,\y-4);
							\coordinate (B0) at (\x-4,\y-4);
							\coordinate (B1) at (\x-3,\y-4);
							\coordinate (B2) at (\x-2,\y-4);
							\coordinate (B3) at (\x-1,\y-4);
							\coordinate (B4) at (\x,\y-4);
							\coordinate (B5) at (\x+1,\y-4);
							\coordinate (B6) at (\x+2,\y-4);
							\coordinate (B7) at (\x+3,\y-4);

							\draw[thick] (\x-5,\y-4) -- (\x+5,\y-4);
							\fill (B1) circle (.1);
							\fill (B2) circle (.1);
							\fill (B3) circle (.1);
							\fill (B4) circle (.1);
							\fill (B5) circle (.1);
							\fill (B6) circle (.1);
							\fill (B7) circle (.1);
							
							\hadic{\x-3}{\y+4}{9}{}
							\hadic{\x-1}{\y+4}{7}{}
							\hadic{\x+1}{\y+4}{5}{}
							\hadic{\x+3}{\y+4}{3}{}

							\adic{\x-3}{\y-4}{3}{}
							\adic{\x-1}{\y-4}{5}{}
							\adic{\x+1}{\y-4}{7}{}
							\adic{\x+3}{\y-4}{9}{}
							\draw (B1) .. controls (\x-2.5,\y-3) and (\x-1.5,\y-3) .. (B3);
							\draw (B3) .. controls (\x-.5,\y-3) and (\x+.5,\y-3) .. (B5);
							\draw (B5) .. controls (\x+1.5,\y-3) and (\x+2.5,\y-3) .. (B7);

							\fill (\x,\y-5) node {$D_z^{-\infty} \cdot T$};
						}
					}

					\foreach \x in {12}
					{
						\foreach \y in {-12}
						{
							\coordinate (H1) at (\x-3,\y+4);
							\coordinate (H2) at (\x-1,\y+4);
							\coordinate (H3) at (\x+1,\y+4);
							\coordinate (H4) at (\x+3,\y+4);
							\coordinate (H5) at (\x+5,\y+4);

							\draw[thick] (\x-5,\y+4) -- (\x+5,\y+4);
							\fill (H1) circle (.1);
							\fill (H2) circle (.1);
							\fill (H3) circle (.1);
							\fill (H4) circle (.1);

							\coordinate (B-1) at (\x-5,\y-4);
							\coordinate (B0) at (\x-4,\y-4);
							\coordinate (B1) at (\x-3,\y-4);
							\coordinate (B2) at (\x-2,\y-4);
							\coordinate (B3) at (\x-1,\y-4);
							\coordinate (B4) at (\x,\y-4);
							\coordinate (B5) at (\x+1,\y-4);
							\coordinate (B6) at (\x+2,\y-4);
							\coordinate (B7) at (\x+3,\y-4);

							\draw[thick] (\x-5,\y-4) -- (\x+5,\y-4);
							\fill (B1) circle (.1);
							\fill (B2) circle (.1);
							\fill (B3) circle (.1);
							\fill (B4) circle (.1);
							\fill (B5) circle (.1);
							\fill (B6) circle (.1);
							\fill (B7) circle (.1);
							
							\draw (H2) -- (B1) -- (H3) -- (B3) -- (H4) -- (B5) -- (H5) -- (B7);
							\draw (B1) .. controls (\x-2.5,\y-3) and (\x-1.5,\y-3) .. (B3);
							\draw (B3) .. controls (\x-.5,\y-3) and (\x+.5,\y-3) .. (B5);
							\draw (B5) .. controls (\x+1.5,\y-3) and (\x+2.5,\y-3) .. (B7);

							\fill (\x,\y-5) node {$D_z \cdot T$};
						}
					}

					\foreach \x in {12}
					{
						\foreach \y in {-24}
						{
							\coordinate (H0) at (\x-5,\y+4);
							\coordinate (H1) at (\x-3,\y+4);
							\coordinate (H2) at (\x-1,\y+4);
							\coordinate (H3) at (\x+1,\y+4);
							\coordinate (H4) at (\x+3,\y+4);
							\coordinate (H5) at (\x+5,\y+4);

							\draw[thick] (\x-5,\y+4) -- (\x+5,\y+4);
							\fill (H1) circle (.1);
							\fill (H2) circle (.1);
							\fill (H3) circle (.1);
							\fill (H4) circle (.1);

							\coordinate (B-1) at (\x-5,\y-4);
							\coordinate (B0) at (\x-4,\y-4);
							\coordinate (B1) at (\x-3,\y-4);
							\coordinate (B2) at (\x-2,\y-4);
							\coordinate (B3) at (\x-1,\y-4);
							\coordinate (B4) at (\x,\y-4);
							\coordinate (B5) at (\x+1,\y-4);
							\coordinate (B6) at (\x+2,\y-4);
							\coordinate (B7) at (\x+3,\y-4);

							\draw[thick] (\x-5,\y-4) -- (\x+5,\y-4);
							\fill (B1) circle (.1);
							\fill (B2) circle (.1);
							\fill (B3) circle (.1);
							\fill (B4) circle (.1);
							\fill (B5) circle (.1);
							\fill (B6) circle (.1);
							\fill (B7) circle (.1);
							
							\draw (H2) -- (B-1) -- (H3) -- (B1) -- (H4) -- (B3) -- (H5) -- (B5);
							\draw (B-1) .. controls (\x-4.5,\y-3) and (\x-3.5,\y-3) .. (B1);
							\draw (B1) .. controls (\x-2.5,\y-3) and (\x-1.5,\y-3) .. (B3);
							\draw (B3) .. controls (\x-.5,\y-3) and (\x+.5,\y-3) .. (B5);
							\draw (B5) .. controls (\x+1.5,\y-3) and (\x+2.5,\y-3) .. (B7);

							\fill (\x,\y-5) node {$D_z^2 \cdot T$};
						}
					}

					\foreach \x in {12}
					{
						\foreach \y in {-36}
						{
							\coordinate (H0) at (\x-5,\y+4);
							\coordinate (H1) at (\x-3,\y+4);
							\coordinate (H2) at (\x-1,\y+4);
							\coordinate (H3) at (\x+1,\y+4);
							\coordinate (H4) at (\x+3,\y+4);

							\draw[thick] (\x-5,\y+4) -- (\x+5,\y+4);
							\fill (H1) circle (.1);
							\fill (H2) circle (.1);
							\fill (H3) circle (.1);
							\fill (H4) circle (.1);

							\coordinate (B-1) at (\x-5,\y-4);
							\coordinate (B0) at (\x-4,\y-4);
							\coordinate (B1) at (\x-3,\y-4);
							\coordinate (B2) at (\x-2,\y-4);
							\coordinate (B3) at (\x-1,\y-4);
							\coordinate (B4) at (\x,\y-4);
							\coordinate (B5) at (\x+1,\y-4);
							\coordinate (B6) at (\x+2,\y-4);
							\coordinate (B7) at (\x+3,\y-4);

							\draw[thick] (\x-5,\y-4) -- (\x+5,\y-4);
							\fill (B1) circle (.1);
							\fill (B2) circle (.1);
							\fill (B3) circle (.1);
							\fill (B4) circle (.1);
							\fill (B5) circle (.1);
							\fill (B6) circle (.1);
							\fill (B7) circle (.1);
							
							\hPrufer{\x-3}{\y+4}{3}{}
							\hPrufer{\x-1}{\y+4}{5}{}
							\hPrufer{\x+1}{\y+4}{7}{}
							\hPrufer{\x+3}{\y+4}{9}{}

							\Prufer{\x-3}{\y-4}{9}{}
							\Prufer{\x-1}{\y-4}{7}{}
							\Prufer{\x+1}{\y-4}{5}{}
							\Prufer{\x+3}{\y-4}{3}{}
							\draw (B1) .. controls (\x-2.5,\y-3) and (\x-1.5,\y-3) .. (B3);
							\draw (B3) .. controls (\x-.5,\y-3) and (\x+.5,\y-3) .. (B5);
							\draw (B5) .. controls (\x+1.5,\y-3) and (\x+2.5,\y-3) .. (B7);

							\fill (\x,\y-5) node {$D_z^{+\infty} \cdot T$};
						}
					}

				\end{tikzpicture}
			\end{center}
			\caption{Limits of a sequence of triangulations obtained by applying Dehn twists.}\label{fig:exmplimit}
		\end{figure}

		\begin{lem}\label{lem:limstable}
			Let $(T_n)_{n \geq 0}$ be a sequence of triangulations which stabilises after $N \geq 0$. Let $d$ be a convergence of $(T_n)_{n \geq 0}$ and fix $k \in \Z$ such that $d(n)=k$ for any $n \geq N$. Then 
			$$\lim_n T_n = D_z^k \cdot T_N.$$
		\end{lem}
		\begin{proof}
			Assume that $(T_n)_{n \geq 0}$ stabilises after $N$ and fix a convergence $d$ and an integer $k \in \Z$ such that $d(n)=k$ for any $n \geq N$. Then, $T_n = D_z^{k} T_N$ for any $n \geq N$ and thus by definition $\lim_n T_n = D_z^{k} \cdot T_N$.
		\end{proof}

		\begin{theorem}
			Let $(T_n)_{n \geq 0} \subset \T(C_{p,q})$ be a converging sequence of triangulations. Then $\lim_n T_n \in \b\T(C_{p,q})$.
		\end{theorem}
		\begin{proof}
			If $(T_n)_{n \geq 0}$ stabilises after a certain rank, then it follows from Lemma \ref{lem:limstable} that $\lim_n T_n \in \T(C_{p,q})$. 

			Otherwise $(T_n)$ has a strictly monotone convergence $d$ and by symmetry, we can assume that this convergence is strictly increasing. 
			Let $T_N$ be a blueprint of $(T_n)$. 
			Let $\gamma$ be an arc of $T_N$. If $\gamma$ is peripheral, then $\gamma\in \lim_n T_n$. And if $\gamma$ is bridging, i.e. $\gamma=(m_i,m_j')$ with $m_i$ on the boundary $\d$ and $m_j'$ on the boundary $\d'$, then $\pi_{m_i}$, $\pi_{m_j'}$ are arcs of $\lim_n T_n$. We have to show a) that arcs of $\lim_n T_n$ pairwise do not cross and that b) the number of arcs in $\lim_n T_n$ is correct, i.e. that it is equal to $p+q$.
			
			a) Non-crossing property.\\
			Let $\beta$, $\gamma$ be arcs in $T_N$. In case they are both peripheral, then $D_z^{+\infty} \beta=\beta$ and $D_z^{+\infty}\gamma=\gamma$ and hence $\beta$ and $\gamma$ are non crossing arcs in $\lim_n T_n$. Let $\beta$ be peripheral and $\gamma$ bridging, so $D_z^{+\infty}\gamma=\{\pi_{m_i},\pi_{m_j'}\}$. 
			Without loss of generality, let $\beta$ lie on the boundary $\d$, i.e. $\beta=(m_k,m_l)$ for $k,l\in\{0,\dots,p-1\}$. As the 
			two arcs do not cross in $T_N$ we have that $i \notin[k,l]$ (interval mod $p$). In particular, 
			$\beta=D_z^{+\infty}\beta$ and $\pi_i\in D_z^{+\infty}\gamma$ do not cross. $\beta$ and 
			$\pi_{m_j'}$ clearly do not cross as they live near opposite boundaries. 
			Last, assume that $\beta=(m_i,m_j')$ and $\gamma=(m_k,m_l')$ are both bridging. 
			Then $D_z^{+\infty}\beta \cup D_z^{+\infty}\gamma$ is a collection of Pr\"ufer arcs (of both 
			boundary components). In particular, they do not cross. 

			\noindent
			b) $\lim_n T_n$ has exactly $p+q$ arcs. \\
			First note that the number of peripheral arcs of $T_N$ and of $\lim_n T_n$ is the same. 
			It only remains to see that the number of bridging arcs of $T_N$ is equal to the number of 
			Pr\"ufer arcs in $\lim_n T_n$. It is enough to show this for the case where $T_N$ does not 
			contain any peripheral arcs. By Lemma~\ref{lem:twoarcs} $T_N$ contains a bridging arc. 
			Up to relabeling the marked points, we can assume that $T_N$ contains the arc $m_1,m_1'$. 
			We cut the annulus 
			along this arc and consider the resulting cylindrical region in the plane: 
			\\
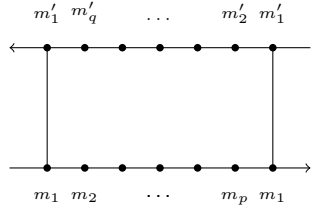
			\begin{center}
				\begin{tikzpicture}[scale = .5]
					\tikzstyle{every node} = [font = \small]
					\foreach \x in {0}
					{
						\foreach \y in {-8}
						{
							\draw[<-] (\x-4,\y+1.6) -- (\x+4,\y+1.6);
							\draw[->] (\x-4,\y-1.6) -- (\x+4,\y-1.6);

							\foreach \t in {-3,-2,...,3}
							{
								\fill (\x+\t,\y+1.6) circle (.1);
								\fill (\x+\t,\y-1.6) circle (.1);
							}

							\fill (\x-3,\y+2) node [above] {\tiny $m_1'$};
							\fill (\x-2,\y+2) node [above] {\tiny $m_q'$};
							\fill (\x ,\y+2) node [above] {\tiny $\cdots$};
				%			\fill (\x+1,\y+2) node [above] {$\cdots$};
							\fill (\x+2,\y+2) node [above] {\tiny $m_2'$};
							\fill (\x+3,\y+2) node [above] {\tiny $m_1'$};

							\fill (\x-3,\y-2) node [below] {\tiny $m_1$};
							\fill (\x-2,\y-2) node [below] {\tiny $m_2$};
				%			\fill (\x-1,\y-2) node [below] {$\cdots$};
							\fill (\x,\y-2) node [below] {\tiny $\cdots$};
							\fill (\x,\y-2) node [below] {\tiny $\cdots$};
							\fill (\x+2,\y-2) node [below] {\tiny $m_p$};
							\fill (\x+3,\y-2) node [below] {\tiny $m_1$};

							\draw (\x-3,\y+1.6) -- (\x-3,\y-1.6);
							\draw (\x+3,\y+1.6) -- (\x+3,\y-1.6);

						}
					}
				\end{tikzpicture}
			\end{center}
			
			It is a polygon with $p+q+2$ vertices. Since all arcs of $T_N$ are bridging, every vertex 
			of this polygon is incident with at least one arc. In particular, $\lim_n T_n$ consists of all 
			the Pr\"ufer arcs and thus has $p+q$ elements. 
			On the other hand, the number of arcs in $T_N$ is $1+(p+q-1)$ as it consists of the arc 
			$(m_1,m_1')$ as well as of all the arcs of a triangulation of a $p+q+2$-gon. 
		\end{proof}

		\begin{rmq}
			Not every (strictly) asymptotic triangulation is the limit of a converging sequence of triangulations. Consider for instance the annulus $C_{1,1}$ with one point $o$ on a boundary component and one point $\iota$ on the other. Then for any triangulation $T \in \T(C_{1,1})$, we have
			$$D_z^{+\infty}\cdot T = \ens{\pi_o,\pi_\iota} \text{ and }D_z^{-\infty}\cdot T = \ens{\alpha_o,\alpha_\iota}$$
			so that 
			$$\ens{\lim_n T_n \ | \ (T_n)_n \subset \T(C_{1,1}) \text{ converges} } = \T(C_{1,1}) \sqcup \ens{\ens{\alpha_o,\alpha_\iota},\ens{\pi_o,\pi_\iota}}$$
			whereas 
			$$\b\T(C_{1,1}) = \T(C_{1,1}) \sqcup \ens{\ens{\alpha_o,\alpha_\iota},\ens{\pi_o,\pi_\iota},\ens{\pi_o,\alpha_\iota},\ens{\alpha_o,\pi_\iota}}.$$

			More generally, for an arbitrary annulus $C_{p,q}$, it follows from the definitions that the limit of a converging sequence of triangulations in $\T(C_{p,q})$ does not contain simultaneously Pr\"ufer curves and adic curves whereas there are asymptotic triangulations with both.
		\end{rmq}

\section*{Acknowledgements}
	The authors would like to thank Aslak Buan and Robert Marsh for helpful discussions. The second author would also like to thank the first author for her kind hospitality during his stay at Karl-Franzens-Universit{\"a}t Graz where this work was initiated.

% \bibliographystyle{amsalpha}
% \bibliography{../biblio}

\begin{thebibliography}{BBM11}

\bibitem[BBM11]{BBM:torsiontubes}
Karin Baur, Aslak Bakke Buan, and Robert Marsh, \emph{Torsion pairs and rigid objects
  in tubes}, to appear in Algebras and Representation Theory, 
  arXiv:1112.6132v1 [math.RT] (2011).

\bibitem[BK04]{BuanKrause:cotilting}
Aslak~Bakke Buan and Henning Krause, \emph{Tilting and cotilting for quivers
  and type {$\widetilde{A}_n$}}, J. Pure Appl. Algebra \textbf{190} (2004),
  no.~1-3, 1--21. \MR{2043318 (2005a:16022)}

\bibitem[Foc97]{Fock:dual}
Vladimir Fock, \emph{Dual Teichm{\"u}ller spaces}, arXiv:dg-ga/9702018 (1997).

\bibitem[FST08]{FST:surfaces}
Sergey Fomin, Michael Shapiro, and Dylan Thurston, \emph{Cluster algebras and
  triangulated surfaces. {I}. {C}luster complexes}, Acta Math. \textbf{201}
  (2008), no.~1, 83--146. \MR{MR2448067 (2010b:57032)}

\bibitem[FT08]{FT:surfaces2}
Sergey Fomin and Dylan Thurston, \emph{Cluster algebras and triangulated
  surfaces. part {II}: {L}ambda lengths}, preprint (2008).

\bibitem[Pen87]{Penner:lambda}
Robert Penner, \emph{The decorated {T}eichm\"uller space of punctured
  surfaces}, Comm. Math. Phys. \textbf{113} (1987), no.~2, 299--339. \MR{919235
  (89h:32044)}

\bibitem[Pen04]{Penner:bordered}
\bysame, \emph{Decorated {T}eichm\"uller theory of bordered surfaces}, Comm.
  Anal. Geom. \textbf{12} (2004), no.~4, 793--820. \MR{2104076 (2006a:32018)}

\bibitem[Rei10]{Reiten:TiltingCluster}
Idun Reiten, \emph{Tilting theory and cluster algebras}, arXiv:1012.6014
  [math.RT] (2010).

\end{thebibliography}

\providecommand{\bysame}{\leavevmode\hbox to3em{\hrulefill}\thinspace}
\providecommand{\MR}{\relax\ifhmode\unskip\space\fi MR }
% \MRhref is called by the amsart/book/proc definition of \MR.
\providecommand{\MRhref}[2]{%
  \href{http://www.ams.org/mathscinet-getitem?mr=#1}{#2}
}
\providecommand{\href}[2]{#2}

\end{document}